\documentclass[letterpaper,12pt]{article}
\usepackage[latin2, utf8x]{inputenc}
\usepackage{amsfonts}
\usepackage{amsthm}
\usepackage{amsmath}

\usepackage[symbol, perpage]{footmisc}

\newtheorem*{n_def}{Definition}
\newtheorem*{n_prop}{Corollary}
\newtheorem{prop}{Proposition}
\newtheorem*{theorem}{Theorem}

\title{On the classification of irrational numbers}
\author{José de Jes\'us Hern\'andez Serda}
\date{May 2015}

\begin{document}

 \maketitle
 
 \begin{abstract}
  In this note we make a comparison between the arithmetic properties of irrational numbers and their dynamical properties under the Gauss map. 
  We show some equivalences between different classifications of irrational numbers such as the Diophantine classes and numbers admitting approximations by rational numbers at a given \emph{speed}.
  We also show that irrational numbers with finite \emph{upper} Lyapunov exponent for the Gauss map satisfy a Diophantine condition.
 \end{abstract}

 Let $\alpha$ be an irrational number. An approximation to $\alpha$ is a sequence of rational numbers converging to $\alpha$. We recall the notions of the speed of an approximation and the distance of a given number to the rational numbers. 
 
 \begin{n_def}
  Let $\psi: \mathbb{N} \rightarrow \mathbb{R}^{+}$ be a function. We say an irrational number $\alpha$ is approximated at \emph{speed} $\psi$ if and only if there exists infinitely many rational numbers $p/q$ such that \[ \left| \alpha - \frac{p}{q}\right| \leq \psi(q).\]
 \end{n_def}

 We have the following classification: We say an irrational number $\alpha$ is, 
 \begin{itemize}
  \item \emph{slow} if and only if there exists $\varepsilon > 0$ such that $\alpha$ is not approximated at speed $\psi(q) = \varepsilon q^{-2}$,
  \item \emph{fast} if and only if there exists $\gamma > 0$ such that $\alpha$ is approximated at speed $\psi(q) = q^{-(2+\gamma)}$ and
  \item \emph{super fast} if and only if $\alpha$ is approximated at speed $\psi(q) = q^{-(2+\gamma)}$ for all $\gamma > 0$.
 \end{itemize}
 
 The arithmetic property of our interest is the speed of the approximations that an irrational number admits.
 If a number $\alpha$ is approximated at speed $\psi_0$ it is also approximated at speed $\psi$ for any $ \psi \geq \psi_0$.
 Given an irrational number $\alpha$ we can compute the function $\phi_{0}(q) = \min_{p \in  \mathbb{Z}} \lbrace |\alpha - p/q|\rbrace$.
 Notice that $\alpha$ may be approximated at some speed $\psi \leq \phi_0$ as long as we have $\psi(q) = \phi_0(q)$ for infinitely many $q$, but $\alpha$ cannot be approximated at speed $\psi$ for any $\psi < \phi_0$.
  
 \begin{n_def}
  Let $\varphi: \mathbb{N} \rightarrow \mathbb{R}^{+}$ be a function. We say an irrational number $\alpha$ is at \emph{distance} $\varphi$ of the rational numbers if and only if for all rational numbers $p/q$ we have \[ \varphi(q) < \left| \alpha - \frac{p}{q} \right|.\]
 \end{n_def}

 With this definition, a number $\alpha$ will be at distance $\varphi$ of the rational numbers for any $\varphi < \phi_0$. 
 We state the usual Diophantine condition in terms of \emph{distance} as follows. We say an irrational number $\alpha$ is, 
 \begin{itemize}
  \item \emph{badly approximable} if and only if there exists $\varepsilon > 0$ such that $\alpha$ is at distance $\varphi(q) = \varepsilon q^{-2}$ of the rational numbers,
  \item \emph{Diophantine of class} $D(\gamma, \varepsilon)$ if and only if there exist $\varepsilon > 0$ and $\gamma \geq 0$ such that $\alpha$ is at distance $\varphi(q) = \varepsilon q^{-(2+\gamma)}$ of the rational numbers.
 \end{itemize}
 The sets
 \[ D(\gamma) = \bigcup_{\varepsilon > 0} D(\gamma, \varepsilon) ,\hspace*{1em} D(+) = \bigcap_{\gamma > 0} D(\gamma) \hspace*{1em}\text{and}\hspace*{1em} D(\infty) = \bigcup_{\gamma > 0} D(\gamma),\]
 are the \emph{Diophantine numbers of class $D(\gamma)$}, the \emph{Diophantine numbers of all classes} and the \emph{Diophantine numbers}, respectively.
 Under these definitions, the class of badly approximable numbers is exactly the Diophantine class $D(0)$.
 For any $0 < \gamma_1 < \gamma_2$ we have that the Diophantine classes satisfy $D(0) \subset D(+) \subset D(\gamma_1) \subset D(\gamma_2) \subset D(\infty)$. 
 All these contentions are strict.
 For example, on \cite{Bugeaud}, Bugeaud shows that there is a Cantor set of numbers of class $D(\gamma_2)$ which are not of class $D(\gamma_1)$ for any $0 < \gamma_1 < \gamma_2$.
 The classifications by speed of approximation and Diophantine condition are complementary in the following sense:
 
 \begin{prop}
 \label{equivalence_diophantine_speed}
  An irrational number is of class $D(+)$ if and only if it is not a fast number. 
  Also, an irrational number is of class $D(\infty)$ if and only if it is not a super fast number.  
 \end{prop}
 \begin{proof}%[Proof of proposition \ref{equivalence_diophantine_speed}]
  Both affirmations are implied by the following: 
  Given $\gamma > 0$, either infinitely many rational numbers satisfy $|\alpha - p/q| \leq q^{-(2+\gamma)}$ or there exists $\varepsilon > 0$ such that all rational numbers satisfy $\varepsilon q^{-(2+\gamma)} < |\alpha - p/q|$, but not both. 
  
  Suppose there is only finitely many rational numbers such $|\alpha - p_i/q_i| \leq q_i^{-(2+\gamma)}$, for each of these numbers we can choose $\varepsilon_i > 0$ such that $\varepsilon_i q_i^{-(2+\gamma)} < |\alpha - p_i/q_i|$.
  Taking $\varepsilon = \min \varepsilon_i$ we have $\varepsilon q^{-(2+\gamma)} < |\alpha - p/q|$ for all rational numbers.  
 \end{proof}

 Liouville's theorem states that algebraic numbers are all of class $D(\infty)$, in particular, that algebraic numbers of degree $d$ are of class $D(d)$. 
 (See, for example \cite{Khinchin}, Theorem 27.)
 The first known examples of transcendental numbers were also given by Liouville, who proved that certain numbers, such as $\sum_{k=0}^{\infty} 10^{-k!}$, are not of class $D(\infty)$. 
 For this reason super fast numbers are also called Liouville numbers. 
 However, there are transcendental numbers of class $D(\infty)$; in his note \cite{Mahler}, Mahler showed that the number $\pi$ is of class $D(42)$.
 Roth \cite{Roth} improved Liouville's theorem showing that algebraic numbers are of class $D(+)$. An example of a transcendental number of class $D(+)$ is Euler's number $e$.
 \[ \ast \ast \ast \]
 
 For the proof of the properties of the continued fraction expansion that we mention here we refer to Khinchin's monograph \cite{Khinchin}. Irrational numbers have an unique continued fraction expansion: \[\alpha = [a_0 : a_1, a_2, ...] = a_0 +\cfrac{1}{a_1 + \cfrac{1}{a_2 + \cfrac{1}{\ddots}}},\] where $a_0 \in \mathbb{Z}$ and $a_n \in \mathbb{N}$ for $n \geq 1$.
 The continued fraction expansion of an irrational number is infinite, whereas rational numbers have finite but not unique continued fraction expansions. Given an irrational number $\alpha = [a_0: a_1, a_2, ...]$ the $n$-th convergent is the rational number $p_n / q_n = [a_0: a_1, ..., a_n]$. The sequence $\lbrace p_n / q_n \rbrace$ converges to $\alpha$ and this sequence is the \emph{best} approximation to $\alpha$: for any other rational number $a/b$ such that $b \leq q_n$ we have \[\left| \alpha - \frac{p_n}{q_n} \right| < \left| \alpha - \frac{a}{b} \right|.\] 
 
 The sequences $p_n$ and $q_n$ follow the recurrence relations $p_n = a_n p_{n-1} + p_{n-2}$ and $q_n = a_n q_{n-1} + q_{n-2}$; under the convention $p_{-1} = 1$, $q_{-1} = 0$ they are valid for $n \geq 1$.
 These relations yield the invariant $q_n p_{n-1} - p_n q_{n-1} = (-1)^{n}$, which in particular shows that convergents are all in lowest terms. 
 For $n \geq 1$ we also have \[ \frac{1}{2 q_{n+1}^2} < \frac{1}{q_n (q_n + q_{n+1})} \leq \left| \alpha - \frac{p_n}{q_n} \right| \leq \frac{1}{q_n q_{n+1}} < \frac{1}{q_n^2},\] which shows that all irrational numbers are approximated at speed $\psi(q) = q^{-2}$, so no irrational number is at any distance $\varphi(q) > q^{-2}$.
 
 %As we mentioned before, all irrational numbers are approximated at speed $\psi(q) = q^{-2}$.
 In order to distinguish different classes of numbers by the speed of approximation we must consider functions $\psi(q) \leq q^{-2}$. 
 The existence of irrational numbers approximated at a given speed is guaranteed by the continued fraction expansion. Let $\psi(q) \leq q^{-2}$ be a positive function and $r = [a_0: a_1, ..., a_k]$ a rational number. 
 Consider the sequence $\lbrace a_n \rbrace$ beginning with the coefficients of the continued fraction of $r$ and choosing $a_{n+1} \geq (q_n^2 \psi(q_n) )^{-1}$ for $n \geq k$. 
 The irrational number $\alpha = [a_0: a_1, a_2, ...]$ satisfies \[ \left| \alpha - \frac{p_n}{q_n} \right| < \frac{1}{q_n q_{n+1}} = \frac{1}{q_n (a_{n+1} q_n + q_{n-1})} \leq \frac{1}{a_{n+1} q_n^2} < \psi(q_n).\]
 Since there are infinitely many options for choosing $a_{n+1}$ at each step, there are uncountably many irrational numbers approximated at a given speed.
 
 We can classify irrational numbers by the properties of the continued fraction expansion, for example, an irrational number $\alpha = [a_0 : a_1, a_2, ...]$ is of \emph{bounded type} if there exists $M > 0$ such that $a_n \leq M$ for all $n$. For this class of numbers we have the well known characterization:
 
 \begin{prop}
 \label{bounded_slow_bad}
  For an irrational number $\alpha$ the following are equivalent:
  a) $\alpha$ is of bounded type,
  b) $\alpha$ is slow,
  c) $\alpha$ is badly approximable.
 \end{prop}

 \begin{proof}%[Proof of proposition \ref{bounded_slow_bad}]
  The equivalence of badly approximable numbers and slow numbers is direct from the definitions.
  We prove that numbers of bounded type are exactly slow numbers. 
  This also follows from Lemma 2.4 in \cite{Urbanski}.
  Let $\alpha = [a_0: a_1, a_2, ...]$ be an irrational number.
  For the convergents of the continued fraction expansion we have \[ \frac{1/3}{a_{n+1} q_n^2} < \left|\alpha - \frac{p_n}{q_n}\right| < \frac{1}{a_{n+1} q_n^2}.\]
  If $\alpha$ is not a slow number then it is approximated at speed $\psi(q) = \varepsilon q^{-2}$ for all $\varepsilon > 0$.
  This implies that for all $\varepsilon > 0$ there are infinitely many $a_{n+1} > (3 \varepsilon)^{-1}$, i.e. $\alpha$ is not of bounded type.
  On the other hand, if $\alpha$ is not a bounded type number then for all $\varepsilon>0$ there are infinitely many $n$ such that $a_{n+1}^{-1} q_n^{-2} \leq \varepsilon q_n^{-2}$, i.e. $\alpha$ is not a slow number.
 \end{proof}
 
 It is also known that irrational numbers with eventually periodic continued fraction expansion are exactly irrational algebraic numbers of degree 2, but it is still an open question whether algebraic numbers of degree $\geq 3$ are of bounded type or not. 
 
 As an example, for the number $(\sqrt{5} + 1)/2 = [1: 1,1,1,...]$ the sequence of denominators $q_n$ is the Fibonacci sequence and its asymptotic growth is $q_n \geq ((\sqrt{5} + 1)/2)^n$.
 For all irrational numbers the recurrence relation $q_{n+1} = a_n q_n + q_{n-1}$ implies that the asymptotic growth of the sequence $q_n$ is at least the same as the Fibonacci sequence, so the series $\sum  q_n^{-1}$ and $\sum (\log q_n)/ q_n$ always converge: 
 \[ \sum_{n=0}^{\infty} \frac{1}{q_n} < \sum_{n=0}^{\infty} \left( \frac{2}{\sqrt{5} + 1} \right)^n = \frac{3 + \sqrt{5}}{2} ,\]
 \[ \sum_{n=3}^{\infty} \frac{\log q_n}{q_n} < \log \left( \frac{\sqrt{5}+1}{2} \right) \sum_{n=3}^{\infty} n \left( \frac{2}{\sqrt{5} + 1} \right)^n = \frac{3 \sqrt{5} - 1}{2} \log \left( \frac{\sqrt{5}+1}{2} \right) .\]
  
 Another classification of irrational numbers comes from the theory of holomorphic functions and it is given by the  Brjuno condition. 
 An irrational number $\alpha$ satisfies the Brjuno condition if and only if \[ \mathcal{B}(\alpha) := \sum_{n=0}^{\infty}  \frac{\log q_{n+1}}{q_n} < \infty,\] where $q_n$ is the denominator of the $n$-th convergent of $\alpha$. 
 %It is known that all Diophantine numbers satisfy the Brjuno condition and that there are numbers, such as the one given by $[a_n = 10^{n!}]$, that are not of class $D(\infty)$ but satisfy the Brjuno condition. 
 
 Recall that $\alpha$ is of class $D(\gamma)$ for some $\gamma > 0$ if and only if there exists $C>0$ such that the sequence of denominators satisfies $q_{n+1} < C q_n^{1+\gamma}$, in this case we have that \[\sum \frac{\log q_{n+1}}{q_n} < \log C \left( \sum \frac{1}{q_n} \right) + (1+\gamma) \left( \sum \frac{\log q_n}{q_n} \right) < \infty.\] 
 Therefore, all numbers of class $D(\infty)$ are Brjuno numbers. But there are numbers, such as $[ a_n = 10^{n!}]$, which satisfy the Brjuno condition but are not of class $D(\infty)$.
 
 \[ \ast \ast \ast \]

 The Gauss Map, $G(x) = (1/x) \mod{1}$, acts as a shift on the continued fraction of irrational numbers in the unit interval: $G([0 : a_1, a_2, ...]) = [0 : a_2, a_3, ...]$. 
 For an irrational number $\alpha \in [0,1]$ the \emph{upper} Lyapunov exponent is defined as \[ \lambda^{+}(\alpha) = \limsup_{n \rightarrow \infty} \frac{1}{n} \log \left|(G^n)' (\alpha)\right| = \limsup_{n \rightarrow \infty} \frac{1}{n} \sum_{j=0}^{n-1} \log \left|G'(G^j (\alpha))\right|,  \] and we have the equivalent expressions (See, for instance \cite{Pollicott}, Eq. 10.) \[ \lambda^{+} (\alpha) = - \limsup_{n \rightarrow \infty} \frac{1}{n} \log \left| \alpha - \frac{p_n}{q_n} \right| = \limsup_{n \rightarrow \infty} \frac{2}{n} \log q_n.\]
 
 Denote by $\mathcal{M}_G$ the set of Borel probability measures on the unit interval which are invariant under the Gauss Map and the associated Lyapunov exponent for measures $\mu \in \mathcal{M}_G$, as \[ \lambda_{\mu} = \int \log |G'| d\mu. \] 
 
 Let $\mathbf{a}_1 : [0,1] \rightarrow \mathbb{N}$ be the function given by $\mathbf{a}_1 ([0; a_1, a_2 ,...]) = a_1$.
 We define for irrational numbers $\alpha \in [0,1]$ and measures $\mu \in \mathcal{M}_G$ the quantities \[ \kappa^{+}(\alpha) = \limsup_{n \rightarrow \infty} \frac{1}{n} \sum_{j=0}^{n-1} \log |\mathbf{a}_1(G^j (\alpha))| \hspace*{1em} \text{  and } \hspace*{1em} \kappa_{\mu} = \int \log \mathbf{a}_1 d \mu .\]  
 We have that $\log \mathbf{a}_1(\alpha) \leq \log|G'(\alpha)| \leq 2 \log(\mathbf{a}_1(\alpha) + 1)$, and therefore $\lambda^{+}(\alpha) \geq \kappa^{+}(\alpha)$ for all irrational numbers and $\lambda_{\mu} \geq \kappa_{\mu}$ for all measures $\mu \in \mathcal{M}_G$. The number $\kappa^{+}(\alpha)$ is also known as the upper \emph{Khinchin exponent} of $\alpha$, for more details on the properties of both exponents see \cite{Fan} and \cite{Pollicott}.
 
 The number $\lambda^{+}(\alpha)/2$ is an upper bound for the exponential growth rate of the sequence $q_n$, and there is also a relation between Lyapunov exponents and the summability of the sequence $a_n$, and this relation allows us to prove the following.

 \begin{theorem}
 \label{Lyapunov_Diophantine}
  If an irrational number is of finite upper Lyapunov exponent then it is of class $D(+)$.
  Furthermore, if an irrational number is of bounded type then it is of finite upper Lyapunov exponent.
 \end{theorem}
 
\begin{proof}%[Proof of proposition \ref{Lyapunov_Diophantine}]
  Let $\alpha$ be an irrational number such that $\lambda^{+}(\alpha) < \infty$ and suppose $\alpha$ is not of class $D(+)$. There exists $\gamma > 0$ and $K>0$ such that $a_{n+1} \geq K q_n^{\gamma} \geq K ((\sqrt{5}+1)/2)^{n\gamma} $ for infinitely many $n$.
  Consequently
  \begin{equation*}
   \begin{split}
     \lambda^{+}(\alpha) \geq \kappa^{+}(\alpha) &= \limsup_{n \rightarrow \infty} \frac{1}{n} \sum_{j=0}^{n-1} \log a_{j+1} \\ &\geq \log K + \frac{\gamma}{2} \log \left( \frac{\sqrt{5} + 1}{2} \right)  \limsup_{n \rightarrow \infty} (n+1) = \infty,
   \end{split}
  \end{equation*}
  which is a contradiction. 
  Therefore $\alpha$ is of class $D(+)$.
  On the other hand, if $\alpha$ is a bounded type number then \[ \lambda^{+} (\alpha) \leq \limsup_{n \rightarrow \infty} \frac{1}{n} \sum_{j=0}^{n-1} \log (a_{j+1} + 1) \leq \limsup_{n \rightarrow \infty} \frac{1}{n} \sum_{j=0}^{n-1} \log (M + 1) = \log (M+1),\] where $M$ is some bound on the coefficients $a_n$.
 \end{proof}
 
 Irrational numbers of finite upper Lyapunov exponent lie strictly in between numbers of bounded type and the class $D(+)$.
 An example is Euler's constant $e$, which is of class $D(+)$ and not of bounded type and satisfies $\lambda^{+}(e) = \infty$:
 The continued fraction expansion of $e$ is \[ [2: 1, 2, 1,1,4,1,1,6,1,1,8,1,1,10,1,1,12,1,...]. \] 
 See \cite{Cohn} for more details. 
 In order to check $e \in D(+)$ notice that for any given $\gamma > 0$ we have the inequalities $(2 q_n^{2+\gamma})^{-1} < (3 a_{n+1} q_n^{2})^{-1} < |e - p_n/q_n|$ for $n \geq 2/\gamma$. 
 Moreover, considering the subsequence $n= 3k+2$ we have,
 \begin{equation*}
  \begin{split}
   \lambda^{+}(e) &=  \limsup_{n \rightarrow \infty} \frac{2}{n} \log q_n \\
   &\geq  \limsup_{k \rightarrow \infty} \frac{2}{3k+2} \log \left( \prod_{j=1}^{k} 2(j+1) \right) \\ 
   &= \limsup_{k \rightarrow \infty} \frac{2 \log \left( 2^k (k+1)! \right)}{3k+2} = \infty.
  \end{split}
 \end{equation*}
 
 It is known that quadratic algebraic numbers are of bounded type and therefore of finite upper Lyapunov exponent; for algebraic numbers of degree $\geq 3$ we do not know whether their upper Lyapunov exponent is finite or not, and to determine the Lyapunov exponent of an algebraic number might be as hard as to find a closed form of its continued fraction.
 
 Consider now the usual Lyapunov exponent for irrational numbers, $\lambda(\alpha)$, which follows the same formulas as $\lambda^{+}(\alpha)$ only switching the $\limsup$ for a $\lim$, and consider the family of probability measures on the unit interval which are invariant under the Gauss map $\mathcal{M}_{G}$. 
 By Birkhoff's Ergodic Theorem we have for an ergodic measure $\mu \in \mathcal{M}_G$ that \[ \lambda_\mu = \int \log |G'| d \mu = \lim_{n \rightarrow \infty} \frac{1}{n} \sum_{j=0}^{n-1} \log |G' (G^j (\alpha))| = \lambda(\alpha) \hspace*{1em} \mu\text{-a.e.},\] where $\lambda(\alpha)$ is the usual Lyapunov exponent.
 The previous theorem and the Ergodic Decomposition Theorem imply the the following.
 (For more details on Ergodic Theory see, for example, \cite{Mane}.)
 
 \begin{n_prop}[\cite{Urbanski} Theorem 2.1]
 \label{Lyapunov_measure}
  If a measure $\mu \in \mathcal{M}_G$ is such that $\lambda_{\mu} < \infty$ then $\mu$-almost every irrational number is of class $D(+)$.
 \end{n_prop}

 For example, the only invariant measure with respect to the Gauss map and equivalent to Lebesgue measure is the called Gauss measure $\mu_g$, it is ergodic and is given by \[ \mu_g (E) =  \frac{1}{\log 2}\int_E \frac{dx}{1+x}.\] 
 The Lyapunov exponent of Gauss measure is $\lambda_{\mu_g} = \pi^2 / 6 \log 2$ and hence the class $D(+)$ has full Lebesgue measure.
 It is known that the subsets of $D(+)$ given by algebraic numbers and the class of bounded type numbers have null Lebesgue measure, so the set of irrational numbers of class $D(+)$ with $\lambda(\alpha) = \pi^2 / 6 \log 2$ which are neither algebraic nor of bounded type has full Lebesgue measure.

\end{document}